\documentclass[11pt]{amsart}
\usepackage{amsmath,enumerate,amsfonts,amssymb,amsthm, verbatim}

\newtheorem{theorem}{Theorem}
\newtheorem{lemma}[theorem]{Lemma}

\newtheorem{corollary}[theorem]{Corollary}
\newtheorem{conjecture}[theorem]{Conjecture}
\newtheorem{maintheorem}[theorem]{Main Theorem}

\theoremstyle{definition}

\newtheorem{notrems}[theorem]{Notation and Remarks}
\newtheorem{notation}[theorem]{Notation}
\newtheorem{remark}[theorem]{Remark}

\newcommand{\Section}[1]{\section{#1}\setcounter{theorem}{0}}

\begin{document}

\title[On the corner contributions to the heat coefficients]{On the corner contributions
to the heat coefficients of geodesic polygons}
\author{Dorothee Schueth}
\address{Institut f\"ur Mathematik, Humboldt-Universit\"at zu
Berlin, 10099 Berlin, Germany}
\email{schueth@math.hu-berlin.de}


\dedicatory{This paper is dedicated to the memory of Marcel Berger.}

\begin{abstract}
Let $\mathcal O$ be a compact Riemannian orbisurface. We compute formulas for
the contribution of cone points of~$\mathcal O$
to the coefficient at~$t^2$ of the asymptotic expansion of the heat trace of~$\mathcal O$,
the contributions at $t^0$ and $t^1$ being known from the literature. As an application,
we compute the coefficient at~$t^2$ of the contribution of interior angles of the form $\gamma=\pi/k$ in
geodesic polygons in surfaces to the asymptotic expansion of the Dirichlet heat kernel of the polygon,
under a certain symmetry assumption locally near the corresponding corner. The main novelty here is
the determination of the way in which the Laplacian of the Gauss curvature at the corner point enters into the
coefficient at~$t^2$.
We finish with a conjecture concerning the analogous contribution of an arbitrary
angle~$\gamma$ in a geodesic polygon.
\end{abstract}

\maketitle

\Section{Introduction}
\label{sec:intro}

\noindent
This paper concerns the influence of certain singularities on the heat coefficients.
The systematic study of heat coefficients in the context
of smooth Riemannian manifolds started in the 1960s. 
Let $(M^d,g)$ be a closed and connected Riemannian manifold, $\Delta_g=-\operatorname{div}_g\circ\operatorname{grad}_g$ the
associated Laplace operator, and $H:(0,\infty)\times M\times M\to{\mathbb R}$ the corresponding
heat kernel. Minakshisundaram and Pleijel~\cite{MP} proved that there is an asymptotic expansion
$$H(t,p,q)\sim_{t\searrow0}(4\pi t)^{-d/2}e^{-\operatorname{dist}^2(p,q)/4t}\sum_{\ell=0}^\infty {\boldsymbol u}_\ell(p,q)t^\ell
$$
for $(p,q)$ in some neighborhood of the diagonal in $M\times M$, and they gave recursive formulas
for the functions~${\boldsymbol u}_\ell$\,.
Correspondingly, the heat trace
$$Z:t\mapsto\int_M H(t,p,p)\operatorname{\,\textit{dvol}}_g(p)=\sum_{j=0}^\infty e^{-t\lambda_j},
$$
where  $0=\lambda_0<\lambda_1\le\lambda_2\le\ldots\to\infty$ is the eigenvalue spectrum
of~$\Delta_g$ with multiplicities,  admits the asymptotic expansion
\begin{equation*}
Z(t)\sim_{t\searrow0}(4\pi t)^{-d/2}\sum_{\ell=0}^\infty a_\ell t^\ell
\end{equation*}
with the so-called \emph{heat coefficients}
$$a_\ell:=\int_M {\boldsymbol u}_\ell(p,p)\operatorname{\,\textit{dvol}}_g(p)\,.
$$
Each of the coefficients $a_\ell$ in this expansion is a spectral invariant
in the sense that it is determined by the eigenvalue spectrum of~$\Delta_g$\,.
Here, ${\boldsymbol u}_0=1$ and $a_0$ is just the volume of~$(M,g)$.

Formulas for $a_1$ and $a_2$ -- more precisely, even for ${\boldsymbol u}_1(p,p)$ and ${\boldsymbol u}_2(p,p)$ --
were first given by Marcel Berger in his announcement~\cite{Be1} of~1966.
One has
$${\boldsymbol u}_1(p,p)=\frac16\operatorname{scal}_g(p)\,,
$$
where $\operatorname{scal}_g$ denotes the scalar curvature associated with~$g$.
Although Berger called that formula ``folklore'',
he was the first to publish a proof of it, in 1968, in his paper~\cite{Be2}. In the same paper,
he proved the formula
$${\boldsymbol u}_2(p,p)
=\frac1{360}(5\operatorname{scal}_g^2-2\|\operatorname{ric}_g\|^2+2\|R_g\|^2-12\Delta_g\operatorname{scal}_g)(p),
$$
where $\operatorname{ric}_g$ and $R_g$ denote
the Ricci and the Riemannian curvature tensor, respectively.
This formula was considerably more intricate to derive than that for~${\boldsymbol u}_1(p,p)$. Berger's method was a direct
calculation in local coordinates, using Minakshisundaram/Pleijel's recursive formulas for the~${\boldsymbol u}_\ell$\,. Meanwhile, in 1967,
McKean and Singer~\cite{MS} had found a shorter way of deriving the corresponding formula for~$a_2$\,. However,
this did not provide an alternative proof of Berger's full formula for ${\boldsymbol u}_2(p,p)$ (which will actually be
needed in the present paper): Its last term is not visible in~$a_2$ since the integral over
$\Delta_g\operatorname{scal}_g$ vanishes.

In 1971, Sakai computed $a_3$ using an approach much similar to Berger's. Later, Gilkey computed formulas
for heat coefficients in more general contexts like Schr\"odinger operators on vector bundles
and, together with Branson, for manifolds with smooth boundary (see~\cite{Gi}, \cite{BG}).
For nonempty boundary, also half-powers of~$t$ can occur in the asymptotic expansion of the corresponding
heat trace (with, e.g., Dirichlet or Neumann boundary conditions).
On the other hand, also surfaces with corners -- albeit only in the case of polygons in euclidean~${\mathbb R}^2$ --
were considered as early as 1966 in Kac's famous paper~\cite{Ka}, where it was shown that the Dirichlet heat
trace satisfies
\begin{equation}
\label{eq:Kac-MS}
Z(t)=(4\pi t)^{-1}\operatorname{vol}(M)-(4\pi t)^{-1/2}\cdot\frac14\operatorname{vol}(\partial M)
     +\sum_{i=1}^N \frac{\pi^2-\gamma_i^2}{24\gamma_i\pi}+O(t^\infty)
\end{equation}
for $t\searrow0$, where $\gamma_1,\dots,\gamma_N$ are the interior angles of the polygon. Actually,
Kac's formula for the angle contribution was more complicated; McKean and Singer brought it into the above form
in their paper~\cite{MS} of 1967, using an unpublished formula of D.~Ray. A full proof of~(\ref{eq:Kac-MS}) was given
in 1988 by van den Berg and Srisatkunarajah~\cite{vdBS}. In 2005, Watson~\cite{Wa} computed the heat coefficients
for geodesic polygons in the round two-sphere; in 2017, U\c car~\cite{Uc} achieved the same for the more
difficult case of geodesic polygons
in the hyperbolic plane. Here, in contrast to the flat case, the asymptotic
expansion of $Z(t)$ does not break off as in~(\ref{eq:Kac-MS}), and there are infinitely many coefficients involving
contributions from the corners. More precisely, for a geodesic polygon in a surface of constant curvature~$K$,
the contribution of an interior angle~$\gamma$ to the small-time asymptotic expansion of~$Z(t)$ has the form
\begin{equation}
\label{eq:Uc}
\sum_{\ell=0}^\infty e_\ell(\gamma)K^\ell t^\ell\,;
\end{equation}
see Corollary~3.37 in~\cite{Uc}, including explicit formulas for the $e_\ell(\gamma)$. As an application,
U\c car proved that for constant $K\ne0$, the set of angles of a geodesic polygons,
including multiplicities, is spectrally determined (Theorem~3.40 in~\cite{Uc}).

While (\ref{eq:Uc}) just turned out from Watson's and U\c car's direct computations,
U\c car also gave, in the special case that $\gamma$ is of the form $\gamma=\pi/k$, a conceptual proof
of the fact that the coefficient at $t^\ell$ must be of the form $e_\ell(\gamma)K^\ell$.
Note that this cannot be achieved by just rescaling, since $K$ can be either positive or negative.
For his reasoning, U\c car used a qualitative description -- involving curvature invariants --
by Donnelly~\cite{Do} and Dryden et al.~\cite{DGGW}
concerning the contribution of orbifold singularities to the heat coefficients of Riemannian orbifolds.
He showed that the heat coefficient contributions of a corner with interior angle $\gamma=\pi/k$
in a geodesic polygon of constant curvature with Dirichlet boundary
conditions can be viewed, in a sense, as the difference between the contributions of an orbifold cone point of order~$k$
and a dihedral orbifold singularity with isostropy group of order~$2k$; see p.~142--144 in~\cite{Uc}.
Since those two contributions are, by Donnelly's structural theory, known
to be determined by $\gamma=\pi/k$ and curvature invariants of appropriate order, 
and since the only curvature invariant of order~$2\ell$ in the case of constant curvature
is $K^\ell$, this implies that the coefficients must be of the form $e_\ell(\gamma)K^\ell$ here.

The present paper constitutes a first step into studying corner contributions in the setting of geodesic polygons
in surfaces of \emph{nonconstant} curvature.
Under a certain symmetry assumption around the corresponding corner~$p$ (see (\ref{ass:symmetry}) in~\ref{not:polyecke}),
we show in our {\bf Main Theorem~\ref{mainthm:corner}} that the contribution of an interior
angle of the form $\gamma=\pi/k$ to the small-time asymptotic expansion
of the Dirichlet heat trace of the polygon is of the form
$$\sum_{t=0}^\infty c_\ell(\gamma)t^\ell
$$
with
$$c_0(\gamma)=\frac{\pi^2-\gamma^2}{24\gamma\pi}\,,\;\;
c_1(\gamma)=\Bigl(\frac{\pi^4-\gamma^4}{720\gamma^2\pi}
+\frac{\pi^2-\gamma^2}{72\gamma\pi}\Bigr)K(p),
$$
and
\begin{equation}
\label{eq:c2pik}
c_2(\gamma)=\Bigl(\frac{\pi^6-\gamma^6}{5040\gamma^5\pi}+\frac{\pi^4-\gamma^4}{1440\gamma^3\pi}
+\frac{\pi^2-\gamma^2}{360\gamma\pi}\Bigr)K(p)^2
-\Bigl(\frac{\pi^6-\gamma^6}{30240\gamma^5\pi}+\frac{\pi^4-\gamma^4}{2880\gamma^3\pi}
+\frac{\pi^2-\gamma^2}{360\gamma\pi}\Bigr)\Delta_gK(p),
\end{equation}
with our sign convention $\Delta_g=-\operatorname{div}_g\circ\operatorname{grad}_g$\,.
The coefficient~$c_0(\gamma)$ is not new (see~\cite{MR}); moreover, $c_1(\gamma)$ and the coefficient
at~$K(p)^2$ in (\ref{eq:c2pik}) coincide, of course, with U\c car's corresponding formulas for constant curvature.
The main novelty here is the coefficient at $\Delta_gK(p)$ in~(\ref{eq:c2pik})
which, of course, did not appear
in the constant curvature case. We conjecture that these formulas generalize to the case of arbitrary $\gamma\in(0,2\pi]$
under the assumption that the Hessian of~$K$ at~$p$ is a multiple of the metric (Conjecture~\ref{conj}).

Our strategy for proving the Main Theorem again uses orbifold theory.
For a cone point $\bar p$~of order~$k$
in a closed Riemannian orbisurface $(\mathcal O,g)$ we compute the coefficient $a_2^{(\{\bar p\})}$ at~$t^2$ of its contribution to the
heat trace of~$(\mathcal O,g)$ ({\bf Theorem~\ref{thm:cone}}), the coefficients at $t^0$ and $t^1$ being known from the
literature~\cite{Do}, \cite{DGGW} (see Remark~\ref{rem:cone-a0a1}).
We then show that under the symmetry assumption~(\ref{ass:symmetry}) from~\ref{not:polyecke},
each $c_\ell(\pi/k)$ is just $\frac12$ times the corresponding $a_\ell^{(\{\bar p\})}$
(Remark~\ref{rem:relation}); this implies our Main Theorem~\ref{mainthm:corner}.
In turn, to prove Theorem~\ref{thm:cone} we first compute the coefficient $b_2(\Phi)$
at~$t^2$ in Donnelly's asymptotic expansion of the integral of $H(t,\,.\,,\Phi(\,.\,))$
over a small neighborhood of~$p$ in a surface $(M,g)$,
where $\Phi$ is an isometry of a (slightly bigger) neighborhood whose differential at~$p$ is a rotation by an
angle~$\varphi\in(0,\pi]$ (Theorem~\ref{thm:b2}); we then use a formula from~\cite{DGGW} (see~(\ref{eq:cone-sum})).
For the computation of~$b_2(\Phi)$, we closely follow Donnelly's proof of the existence
of the mentioned asymptotic expansion (in a much more general setting) from~\cite{Do}.
In preparation for that, we have to give expansions for $r\searrow0$
of $r\mapsto{\boldsymbol u}_0(\exp_p(ru),\Phi(\exp_p(ru)))$ (up to order the order of~$r^4$) and of $r\mapsto{\boldsymbol u}_1(\exp_p(ru),\Phi(\exp_p(ru))$
(up to the order of~$r^2$), where $u\in T_pM$ is a unit vector (Lemma~\ref{lem:u0bisr4}).
Moreover, we need the expansion of the Riemannian distance $\operatorname{dist}(\exp_p(ru),\Phi(\exp_p(ru)))$ up to the order of $r^6$
(Corollary~\ref{cor:distexpansion}, Lemma~\ref{lem:distexpansion-uv}). Since a formula for the sixth order expansion of the distance funcion
did not seem to be available in the literature, we first give a general formula for the sixth order
expansion of $\operatorname{dist}^2(\exp_p(x),\exp_p(y))$ in surfaces, where $x,y$ are tangent vectors at~$p$
(Lemma~\ref{lem:distexpansion}). For the proof, we partly follow an approach by Nicolaescu~\cite{Ni} which uses
a Hamilton-Jacobi equation for~$\operatorname{dist}^2(q,\,.\,)$.

This paper is organized as follows:
In Section~\ref{sec:prelims}, we provide some notation and technical preparations, among these the sixth order expansion
of the distance function in surfaces (Lemma~\ref{lem:distexpansion} and Corollary~\ref{cor:distexpansion}; the proof
of Lemma~\ref{lem:distexpansion} is postponed to the Appendix).
In Section~\ref{sec:donnelly}, we first prove Lemma~\ref{lem:u0bisr4} concerning the mentioned expansions of ${\boldsymbol u}_0$ and ${\boldsymbol u}_1$\,;
we then deduce Theorem~\ref{thm:b2} concerning $b_2(\Phi)$ by following Donnelly's approach.
Section~\ref{sec:orbi} is devoted to the computation of $a_2^{(\{\bar p\})}$ for cone points of order~$k$ in orbisurfaces
(Theorem~\ref{thm:cone}), using Theorem~\ref{thm:b2} and Dryden et al.'s formula (\ref{eq:cone-sum}).
In Section~\ref{sec:applic} we prove our Main Theorem~\ref{mainthm:corner}; we conclude with some remarks and Conjecture~\ref{conj}.

\medskip
\bf Acknowledgement. \rm
The author thanks the organizers of the conference ``Riemannian Geometry Past, Present and Future: an homage to Marcel Berger''
in December 2017 for inviting her as a speaker, which was a great honour for her. Part of the inspiration
for the results in this article was provided by having a closer look, for that occasion,
at Berger's seminal early works \cite{Be1}, \cite{Be2},
\cite{Be3}, \cite{BGM} in spectral geometry
-- and also by his fearless use of a bit of ``calcul brutal'' when needed
(quotation from the first line of p.~923 in~\cite{Be2}).

\Section{Preliminaries}
\label{sec:prelims}

\noindent
In this paper, $(M,g)$ will always denote a two-dimensional Riemannian manifold
and $K:M\to{\mathbb R}$ its Gauss curvature. Let $\Delta_g=-\operatorname{div}_g\circ\operatorname{grad}_g$ be the
Laplace operator on smooth functions on~$M$. By $\nabla^2 K$ we denote the Hessian tensor of~$K$;
that is, $\nabla^2 K_p(x,y)=g_p(\nabla_x\operatorname{grad}_g K,y)$ for $x,y\in T_pM$, where $\nabla$ denotes the
Levi-Civita connection.
In particular, if $\{u,\tilde u\}$ is an orthonormal basis of~$T_pM$ then
$$\Delta_gK(p)=-\nabla^2K_p(u,u)-\nabla^2K_p(\tilde u,\tilde u).
$$

\begin{notrems}\
\label{notrems:density}
Let $p\in M$ and $u\in T_pM$ with $\|u\|=1$.

(i)
If $\tilde u\in T_pM$ is a unit vector with $\tilde u\perp u$ and $J$ the Jacobi field
along the geodesic $\gamma_u$ with $J(0)=0$, $J'(0)=\tilde u$, then
$$\ell_u(r):=\|(d\exp_p)_{ru}(r\tilde u)\|=\|J(r)\|
$$
has the following well-known expansion for $r\searrow0$:
\begin{equation}
\label{eq:ell}
\ell_u(r)=r-\frac16 K(p)r^3-\frac1{12}dK_p(u)r^4
+\Bigl(\frac1{120}K(p)^2-\frac1{40}\nabla^2K_p(u,u)\Bigr)r^5+O(r^6).
\end{equation}
This follows from the Jacobi equation $J''=-(K\circ\gamma_u)J$ for Jacobi fields orthogonal to~$\dot\gamma_u$\,.

(ii)
For small $r>0$, we denote by $\theta_u(r)$ the so-called volume density
or area distortion of $\exp_p$ at $ru\in T_pM$. In other words, $\theta_u(r)=(\det g_{ij}(ru))^{1/2}$
in normal coordinates around~$p$. Since $\exp_p$ is a radial isometry and we are in dimension two,
we have
$$\theta_u(r)=\ell_u(r)/r.
$$
Thus (\ref{eq:ell}) implies:
\begin{equation}
\label{eq:theta}
\theta_u(r)=1-\frac16 K(p)r^2-\frac1{12}dK_p(u)r^3
+\Bigl(\frac1{120}K(p)^2-\frac1{40}\nabla^2K_p(u,u)\Bigr)r^4+O(r^5).
\end{equation}

(iii)
For $\ell\in{\mathbb N}_0$, let ${\boldsymbol u}_\ell$ denote the (universal) functions, defined on some neighborhood of the diagonal in $M\times M$,
which in case of closed surfaces appear in the asymptotic expansion of the heat kernel of $(M,g)$:
\begin{equation*}
H(t,p,q)\sim(4\pi t)^{-1}\exp(-\operatorname{dist}^2(p,q)/4t)\cdot\sum_{\ell=0}^\infty {\boldsymbol u}_\ell(p,q)t^\ell\text{ \ as \ }t\searrow0,
\end{equation*}
where $\operatorname{dist}:M\times M\to{\mathbb R}$ denotes Riemannian the distance function of $(M,g)$.

(iv)
It is well-known that ${\boldsymbol u}_0=\theta^{-1/2}$ (see~\cite{MP}); more precisely,
$${\boldsymbol u}_0(p,\exp_p(ru))=\theta_u(r)^{-1/2}
$$
for small $r\ge0$. In particular, (\ref{eq:theta}) implies
\begin{equation}
\label{eq:u0}
{\boldsymbol u}_0(p,\exp_p(ru))=1+\frac1{12}K(p)r^2+\frac1{24}dK_p(u)r^3
+\Bigl(\frac1{160}K(p)^2+\frac1{80}\nabla^2K_p(u,u)\Bigr)r^4
+O(r^5).
\end{equation}

(v)
As proved in~\cite{Be2} by Marcel Berger, the restriction of~${\boldsymbol u}_2$ to the diagonal is given by
$${\boldsymbol u}_2(p,p)=\frac1{72}\operatorname{scal}^2(p)-\frac1{180}\|\operatorname{ric}_p\|^2+\frac1{180}\|R_p\|^2-\frac1{30}\Delta_g\operatorname{scal}(p),
$$
where $\operatorname{scal}$, $\operatorname{ric}$, $R$ denote the scalar curvature, the Ricci and the Riemannian curvature tensor, respectively.
Recall our choice of sign for $\Delta_g=-\operatorname{div}_g\circ\operatorname{grad}_g$\,.
In dimension two, the above formula simplifies to
\begin{equation}
\label{eq:u2}
{\boldsymbol u}_2(p,p)=\frac1{15}K(p)^2-\frac1{15}\Delta_gK(p).
\end{equation}

\end{notrems}

\begin{lemma}
In the notation of~\ref{notrems:density},
\begin{multline}
\label{eq:u1}
{\boldsymbol u}_1(p,\exp_p(ru))=\frac13 K(p)+\frac16 dK_p(u)r
+\Bigl(\frac1{30}K(p)^2-\frac1{120}\Delta_gK(p)+\frac1{20}\nabla^2K_p(u,u)\Bigr)r^2\\
+O(r^3)
\end{multline}
for $r\searrow0$.
\end{lemma}

\begin{proof}
One way to obtain this is specializing Sakai's formulas (3.7), (4.3)--(4.5) from~\cite{Sa} (for arbitrary dimension~$n$)
to dimension two and then translating into our notation. An alternative proof which uses the
two-dimensional setting right away is as follows:
By Minakshisundaram/Pleijel's recursion formula from~\cite{MP} for the ${\boldsymbol u}_\ell$\,, applied to $\ell=1$,
\begin{equation}
\label{eq:u1rec}
{\boldsymbol u}_1(p,\exp_p(ru))=
-{{\boldsymbol u}_0(p,\exp(ru))\int_0^1{\boldsymbol u}_0(p,\exp_p(tru))^{-1}
\bigl(\Delta_g{\boldsymbol u}_0(p,\,.\,)\bigr)(\exp_p(tru))\,dt.}
\end{equation}
For small $r>0$, the curvature of the distance sphere $\partial B_r(p)$ at $\exp_p(ru)$ is
$$\frac1r+\frac{\theta'_u(r)}{\theta_u(r)}=\frac1r-\frac13 K(p)r+O(r^2),
$$
where the latter equation holds by~(\ref{eq:theta}).
Moreover, letting $\tilde u$ be a unit vector orthogonal to~$u$ and
$$u(s):=\cos(s)u+\sin(s)\tilde u,
$$
the curve $c:t\mapsto\exp_p(ru(t/\ell_u(r)))$
satisfies $c(0)=\exp_p(ru)$, $\|\dot c(0)\|=1$ and
$$\Bigl\langle \frac D{dt}\dot c(0),\dot c(0)\Bigr\rangle 
=\frac12\cdot\frac d{dt}\Bigl|_{t=0}\,\ell_{u(t/\ell_u(r))}(r)^2/\ell_u(r)^2.
$$
Using~(\ref{eq:ell}), one can check that the latter expression is of order $O(r^2)$ for $r\searrow0$.
Thus, for any function~$f$ near~$p$ which is of the form
$$f(\exp_p(ru))=\alpha(r)\beta(u)
$$ with smooth $\alpha:[0,\varepsilon)\to{\mathbb R}$
and $\beta:S^1_p\to{\mathbb R}$, where $S^1_p\subset (T_pM,g_p)$ denotes the unit circle,
one has
\begin{multline}
\label{eq:alphabeta}
(\Delta_gf)(\exp_p(ru))=-\Bigl[\alpha''(r)+\Bigl(\frac1r-\frac13 K(p)r+O(r^2)\Bigr)\alpha'(r)\Bigr]\beta(u)\\
-\alpha(r)\Bigl(\frac1{\ell_u(r)^2}\nabla^2\beta_u(\tilde u,\tilde u)+\frac{O(r^2)}{\ell_u(r)}d\beta_u(\tilde u)\Bigr), 
\end{multline}
where $\nabla^2\beta$ here denotes the Hessian of~$\beta$ as a function
on the circle~$S^1_p$\,.
Viewing $u\mapsto dK_p(u)$, $u\mapsto \nabla^2K_p(u,u)$ in formula~(\ref{eq:u0})
as functions on $S^1_p$ (not on $T_pM$), we can apply~(\ref{eq:alphabeta}) to the three nonconstant terms in~(\ref{eq:u0}).
Evaluating up to the order of~$r^2$ gives
$$\bigl(\Delta_g{\boldsymbol u}_0(p,\,.\,)\bigr)(\exp_p(ru))=A_1+A_2+A_3+O(r^3),
$$
where
\begin{align*}
A_1={}&-\frac1{12}K(p)\Bigl(2+2-\frac23 K(p)r^2\Bigr)=-\frac13 K(p)+\frac1{18}K(p)^2r^2,\\
A_2={}&-\frac1{24}\bigl(dK_p(u)(6r+3r)-r\cdot dK_p(u)\bigr)=-\frac13 dK_p(u)r\\
A_3={}&-\Bigl(\frac1{160}K(p)^2+\frac1{80}\nabla^2K_p(u,u)\Bigr)(12r^2+4r^2)
   -\frac1{80}r^2\bigl(2\nabla^2K_p(\tilde u,\tilde u)-2\nabla^2K_p(u,u)\bigr)\\
 ={}&-\Bigl(\frac1{10}K(p)^2+\frac7{40}\nabla^2K_p(u,u)+\frac1{40}\nabla^2K_p(\tilde u,\tilde u)\Bigr)r^2\\
 ={}&\Bigl(-\frac1{10}K(p)^2+\frac1{40}\Delta_g K(p)-\frac3{20}\nabla^2K_p(u,u)\Bigr)r^2.
\end{align*}
Thus,
\begin{multline*}
\bigl(\Delta_g{\boldsymbol u}_0(p,\,.\,)\bigr)(\exp_p(ru))=-\frac13 K(p)-\frac13 dK_p(u)r
+\Bigl(-\frac2{45}K(p)^2+\frac1{40}\Delta_g K(p)-\frac3{20}\nabla^2K_p(u,u)\Bigr)r^2\\
+O(r^3).
\end{multline*}
By this and~(\ref{eq:u0}),
\begin{multline*}
\bigl(\Delta_g{\boldsymbol u}_0(p,\,.\,)/{\boldsymbol u}_0(p,\,.\,)\bigr)(\exp_p(ru))=
-\frac13 K(p)-\frac13 dK_p(u)r\\
+\Bigl(-\frac1{60}K(p)^2+\frac1{40}\Delta_g K(p)-\frac3{20}\nabla^2K_p(u,u)\Bigr)r^2
+O(r^3).
\end{multline*}
The integral in~(\ref{eq:u1rec}) thus gives
\begin{equation*}
-\frac13 K(p)-\frac16 dK_p(u)r
+\Bigl(-\frac1{180}K(p)^2+\frac1{120}\Delta_g K(p)-\frac1{20}\nabla^2K_p(u,u)\Bigr)r^2+O(r^3).
\end{equation*}
Multiplying this by $-{\boldsymbol u}_0(p,\exp_p(ru))=-1-\frac1{12}K(p)r^2+O(r^3)$ (see~(\ref{eq:u0})),
we obtain the desired formula.
\end{proof}

\begin{lemma}
\label{lem:distexpansion}
As above, let $\operatorname{dist}:M\times M\to{\mathbb R}$ be the Riemannian distance function on the surface~$(M,g)$.
Then for all $x,y\in T_pM$,
\begin{equation}
\label{eq:distexpansion}
\begin{split}
\operatorname{dist}^2(\exp_p(x),\exp_p(y))={}&\|x-y\|^2-\frac13 K(p)\|x\wedge y\|^2
							-\frac1{12}dK_p(x+y)\|x\wedge y\|^2\\
              &-\frac1{45}K(p)^2\bigl(\|x\|^2-4\langle x,y\rangle +\|y\|^2\bigr)\|x\wedge y\|^2\\
							&-\frac1{60}\bigl(\nabla^2K_p(x,x)+\nabla^2K_p(x,y)+\nabla^2K_p(y,y)\bigr)\|x\wedge y\|^2\\
							&+o((\|x\|^2+\|y\|^2)^3).
\end{split}
\end{equation}
\end{lemma}

We postpone the proof of Lemma~\ref{lem:distexpansion} to the Appendix.

\begin{corollary}
\label{cor:distexpansion}
Let $u\ne v$ be vectors in the unit sphere $S^1_p\subset T_pM$.
Let $\varphi:=\arccos\langle u,v\rangle \in(0,\pi]$ denote the angle between $u$ and~$v$.
Then, using the abbreviation $C:=\|u-v\|=\sqrt{2-2\cos\varphi}$,
we have
\begin{multline*}
\operatorname{dist}(\exp_p(ru),\exp_p(rv))=C r-\frac{\sin^2\varphi}{6C}K(p)r^3
-\frac{\sin^2\varphi}{24C}dK_p(u+v)r^4\\
-\Bigl[\Bigl(\frac{\sin^4\varphi}{72C^3}
   +\frac{\sin^2\varphi\cdot(2-4\cos\varphi)}{90C}\Bigr)K(p)^2
+\frac{\sin^2\varphi}{120C}\bigl(\nabla^2K_p(u,u)+\nabla^2K_p(u,v)+\nabla^2K_p(v,v)\bigr)\Bigr]r^5\\
-\frac{\sin^4\varphi}{144C^3}K(p)dK_p(u+v)r^6+O(r^7)
\end{multline*}
for $r\searrow0$.
\end{corollary}

\begin{proof}
Note that $\|ru\wedge rv\|^2=r^4\sin^2\varphi$. 
The claimed formula now follows directly by applying Lemma~\ref{lem:distexpansion} to $x:=ru$, $y:=rv$ 
and forming the square root of the resulting power series.
\end{proof}

\Section{Donnelly's $b_2$ for rotations in dimension two}
\label{sec:donnelly}

\begin{notrems}
\label{notrems:rot}
We continue to use the notation of Section~\ref{sec:prelims}; in particular, $(M,g)$ is a two-dimensional
Riemannian manifold.
Let $p\in M$ and $\varphi\in(0,\pi]$. Equip $T_pM$ with an arbitrarily chosen orientation, and let $D^\varphi:T_pM\to T_pM$ denote the
corresponding euclidean rotation by the angle~$\varphi$. Let $\varepsilon_1>0$ such that $\exp_p$ is a diffeomorphism
from $B_{\varepsilon_1}(0_p)\subset T_pM$ to its image $B:=B_{\varepsilon_1}(p)\subset M$. Choose $0<\varepsilon<\varepsilon_2<\varepsilon_1$\,, and let
$$V:=B_{\varepsilon_2}(p)\subset B\text{ \ and \ }U:=B_\varepsilon(p)\subset V.
$$
Suppose that there exists an
isometry
$$\Phi:(B,g)\to (B,g)\text{ \ with }\Phi(p)=p\text{ and }d\Phi_p=D^\varphi.
$$
A result by Donnelly~\cite{Do}, applied to this special situation, says that 
$$I(t):=\int_U H(t,q,\Phi(q))\operatorname{\,\textit{dvol}}_g(q)
$$
admits an asymptotic expansion of the form
\begin{equation}
\label{eq:I-asymptotic}
I(t)\sim\sum_{\ell=0}^\infty b_\ell(\Phi) t^\ell\text{ for }t\searrow0,
\end{equation}
where $H:=H_V$ denotes the (Dirichlet) heat kernel of~$V$.
\end{notrems}

\begin{remark}
\label{rem:b0b1}
Note that no factor $(4\pi t)^{-n/2}$ is visible on the right hand side of~(\ref{eq:I-asymptotic}); this is due
to the fact that the
dimension~$n$ of the fixed point set $\{p\}$ of~$\Phi$ is zero here.
In a much more general situation, involving fixed point sets of arbitrary isometries on manifolds of arbitrary
dimension, Donnelly proved a structural result for analogous coefficients~$b_\ell$ and explicitly computed $b_0$ and $b_1$
(but not $b_2$). In our
above situation, Donnelly's formulas for $b_0$ and~$b_1$ amount to
$$b_0(\Phi)=(2-2\cos\varphi)^{-1}\text{ and }b_1(\Phi)=2K(p)(2-2\cos\varphi)^{-2}
$$
(see also~\cite{DGGW} for this in the case $\varphi\in\{2\pi/k\mid k\in{\mathbb N}\}$). In this section we will compute $b_2(\Phi)$; see
Theorem~\ref{thm:b2}. Our
strategy is to follow Donnelly's general approach from~\cite{Do}, p.~166/167, in our special setting.
\end{remark}

\begin{remark}
\label{rem:feeling}

(i)
The coefficients in~(\ref{eq:I-asymptotic}) will not change if in the definition of~$I(t)$
we replace~$V$ by any other open, relatively compact, smoothly bounded neighborhood
of $\overline U$ in~$M$ (e.g., $M$ itself in case $M$ is a closed surface). In fact, while the individual
values of $H(t,q,w)$ will of course depend on this choice (and so will $I(t)$), the coefficients
of the small-time expansion of $H(t,q,w)$ for $q,w\in U$ do not depend on it.
This is due to the ``Principle of not
feeling the boundary''; see, e.g., \cite{Ka}, \cite{Hs}, or Lemma~3.17 in~\cite{Uc}.

(ii)
The coefficients in~(\ref{eq:I-asymptotic}) will not change, either, if in the definition of~$I(t)$
we replace the integral over $U$ by the integral over any smaller open neighborhood $\tilde U\subset U$ of~$p$. This is due
to the fact that by our choices of $\varepsilon$ and~$\varphi$, 
the function $U\setminus\tilde U:q\mapsto\operatorname{dist}(q,\Phi(q))\in{\mathbb R}$ will be bounded below by some positive constant, which
implies that the integral of $H(t,q,\Phi(q))$ over $U\setminus\tilde U$ vanishes to infinite order as $t\searrow0$.
\end{remark}

\begin{lemma}
\label{lem:distexpansion-uv}
Let the situation be as in~\ref{notrems:rot}. Then we have $dK_p=0$. Moreover, if $\varphi\in(0,\pi)$
then $\nabla^2K_p=-\frac12\Delta_gK(p)\cdot g_p$\,.
Finally, for every $\varphi\in(0,\pi]$ and every $u\in S^1_p$\,, the function
$$d_u:r\mapsto\operatorname{dist}(\exp_p(ru),\exp_p(rv)),
$$
where $v:=D^\varphi(u)$, satisfies
\begin{multline}
\label{eq:distexpansion-uv}
d_u(r)=C r-\frac{\sin^2\varphi}{6C}K(p)r^3\\
-\Bigl[\Bigl(\frac{\sin^4\varphi}{72C^3}+\frac{\sin^2\varphi\cdot(2-4\cos\varphi)}{90C}\Bigr)K(p)^2
-\frac{\sin^2\varphi\cdot(2+\cos\varphi)}{240C}\Delta_g K(p)\Bigr]r^5+O(r^7)
\end{multline}
 for $r\searrow 0$, where $C=\sqrt{2-2\cos\varphi}$.
\end{lemma}

\begin{proof}
The first two statements are clear
since $dK_p$ and $\nabla^2K_p$ are invariant under $D^\varphi$.
In particular, in the case $\varphi\in(0,\pi)$ we have
$$\nabla^2K_p(u,u)+\nabla^2K_p(u,v)+\nabla^2K_p(v,v)
=-\frac12\Delta_gK(p)\cdot(2+\cos\varphi),
$$
so (\ref{eq:distexpansion-uv})
follows by Corollary~\ref{cor:distexpansion}.
In case $\varphi=\pi$, (\ref{eq:distexpansion-uv}) trivially holds
by $d_u(r)=2r$, $C=2$, $\sin\varphi=0$.
\end{proof}

\begin{remark}
In the following Lemma~\ref{lem:u0bisr4}
some formulas would become simpler if we assumed $\nabla^2K_p$ to be a multiple of $g_p$\,.
This would imply $\nabla^2K_p(u,u)=-\frac12\Delta_gK(p)$ for all $u\in S^1_p$\,.
Recall from Lemma~\ref{lem:distexpansion-uv} that this is the case anyway if $\varphi\in(0,\pi)$ in~\ref{notrems:rot}.
For $\varphi=\pi$, however, the above assumption on $\nabla^2 K_p$ would unnecessarily make the Lemma less precise.
\end{remark}

\begin{lemma}
\label{lem:u0bisr4}
In the situation of~\ref{notrems:rot}, letting $C:=\sqrt{2-2\cos\varphi}$ and
$v:=D^\varphi u$ we have
\begin{multline*}
{\boldsymbol u}_0(\exp_p(ru),\exp_p(rv))=1+\frac1{12}K(p)d_u(r)^2+{}\\
\shoveright{+\Bigl(\frac1{24C^2}\nabla^2K_p(u,u)+\frac1{160}K(p)^2-\frac1{120}\nabla^2K_p(u,u)
   \Bigr)d_u(r)^4
+O(d_u(r)^5),}\\
\shoveleft{{\boldsymbol u}_1(\exp_p(ru),\exp_p(rv))=\frac13 K(p)+{}}\\
\shoveright{+\Bigl(\frac1{6C^2} \nabla^2K_p(u,u)+\frac1{30}K(p)^2
  -\frac1{30}\nabla^2K_p(u,u)-\frac1{120}\Delta_gK(p)\Bigr)d_u(r)^2
+O(d_u(r)^3),}\\
\shoveleft{{\boldsymbol u}_2(\exp_p(ru),\exp_p(rv))=\frac1{15}K(p)^2-\frac1{15}\Delta_gK(p)+O(d_u(r)^1).}\\
\end{multline*}
\end{lemma}

\begin{proof}
Let $q(r):=\exp_p(ru)$, $w(r):=\exp(rv)$. Moreover, for small $r\ge0$, let $Y(r)\in T_{q(r)}M$
be the vector with $\exp_{q(r)}(Y(r))=w(r)$. Then $\|Y(r)\|_g=d_u(r)$, $Y(0)=0$, and the initial
covariant derivative of $Y$ is
$$Y'(0)=D^\varphi u -u =(\cos\varphi-1)u+(\sin\varphi)\tilde u=-\frac12C^2u+(\sin\varphi)\tilde u,
$$
where $\tilde u:=D^{\pi/2}u$. We apply~(\ref{eq:u0}) to $q(r)$ instead of~$p$ and $d_u(r)$ instead of~$r$,
and we use $dK_p=0$ (see Lemma~\ref{lem:distexpansion-uv}).
Recalling~(\ref{eq:distexpansion-uv}) and, in particular, $r=O(d_u(r))$ for $r\searrow0$ (since $C>0$), we obtain
\begin{align*}
{\boldsymbol u}_0(q(r),w(r))={}&1+\frac1{12}K(q(r))d_u(r)^2+\frac1{24}dK_{q(r)}(Y(r))d_u(r)^2\\
&+\frac1{160}K(q(r))^2 d_u(r)^4+\frac1{80}\nabla^2K_{q(r)}(Y(r),Y(r))d_u(r)^2
+O(d_u(r)^5)\\
={}&1+\frac1{12}(K(p)+\frac12 r^2\nabla^2K_p(u,u))d_u(r)^2+\frac1{24}r \nabla^2K_p(u,rY'(0))d_u(r)^2\\
&+\frac1{160}K(p)^2d_u(r)^4+\frac1{80}\nabla^2K_p(rY'(0),rY'(0))d_u(r)^2
+O(d_u(r)^5).
\end{align*}
We have
\begin{equation}
\label{eq:hilf}
r\nabla^2K_p(u,rY'(0))=-\frac12\nabla^2K_p(u,u)C^2r^2\text{ \; and \; }
\nabla^2K_p(rY'(0),rY'(0))=\nabla^2K_p(u,u)C^2r^2.
\end{equation}
In case $\pi=\varphi$ this follows from $Y'(0)=-\frac12C^2u+0$ and $C=2$; in case $\varphi\in(0,\pi)$ it follows
from the fact that $\nabla^2K_p$ is a multiple of $g_p$ (see Lemma~\ref{lem:distexpansion-uv}) and from $\|Y'(0)\|_g^2=C^2$.
The first statement of the lemma now follows by noting that $C^2r^2=d_u(r)^2+O(d_u(r)^4)$.
Analogously, (\ref{eq:u1}) and evaluating up the order of~$r^2$ gives, using (\ref{eq:hilf}) again:
\begin{align*}
{\boldsymbol u}_1(q(r),w(r))={}&\frac13 K(q(r))+\frac16 dK_{q(r)}(Y(r))\\
&+\Bigl(\frac1{30}K(q(r))^2-\frac1{120}\Delta_gK(q(r))\Bigr)d_u(r)^2
+\frac1{20}\nabla^2K_{q(r)}(Y(r),Y(r))
+O(d_u(r)^3)\\
={}&\frac13\Bigl(K(p)+\frac12 r^2 \nabla^2K_p(u,u)\Bigr)+\frac16\cdot\Bigl(-\frac12 \nabla^2K_p(u,u)C^2 r^2\Bigr)\\
&+\Bigl(\frac1{30}K(p)^2-\frac1{120}\Delta_gK(p)\Bigr)d_u(r)^2+\frac1{20}\nabla^2K_p(u,u)C^2r^2
+O(d_u(r)^3),
\end{align*}
which implies the second formula. The third formula is clear by~(\ref{eq:u2}).
\end{proof}

\begin{theorem}
\label{thm:b2}
In the situation of~\ref{notrems:rot}, and with
$C:=\sqrt{2-2\cos\varphi}$, the coefficient $b_2(\Phi)$ in~(\ref{eq:I-asymptotic}) is given by
$$b_2(\Phi)=\bigl(\frac{12}{C^6}-\frac2{C^4}\bigr)K(p)^2-\frac2{C^6}\Delta_gK(p).
$$
\end{theorem}

\begin{proof}
Recall the notation of~\ref{notrems:rot}.
There is a neighborhood~$\Omega\subset V\times V$ of the diagonal such that for all $(q,w)\in\Omega$,
$$4\pi t\,e^{\operatorname{dist}^2(q,w)/4t}H(t,q,w)-\sum_{k=0}^2 {\boldsymbol u}_k(q,w)t^k\in O(t^3)\text{ as }t\searrow0,
$$
and this holds locally uniformly on~$\Omega$. By Remark~\ref{rem:feeling}(ii), we can assume that $\varepsilon$ is so small
that $(q,\Phi(q))\in\Omega$ for all $q$ in the closure $\overline U\subset V$ of $U=B_\varepsilon(p)$.
Using polar coordinates on~$U$ and writing
$$\bar H(t,x,y):=H(t,\exp_p(x),\exp_p(y))
$$
for $x,y\in B_{\varepsilon_2}(0_p)$, we have
$$I(t)=\int_{S^1_p}\int_0^\varepsilon \bar H(t,ru,rD^\varphi(u))\cdot\ell_u(r)\,dr\,du,
$$
where $\ell_u$ is as in~\ref{notrems:density}.
Note that by our choices of $\varepsilon$ and~$\varphi$, the function
$$S^1_p\times[0,\varepsilon)\ni (u,r)\mapsto d_u(r):=\operatorname{dist}(\exp_p(u),\exp_p(rD^\varphi(u)))\in{\mathbb R}
$$
is continuous, and it is smooth on $S^1_p\times(0,\varepsilon)$. By Lemma~\ref{lem:distexpansion-uv},
for every $u\in S^1_p$ the function~$d_u$
has the expansion (\ref{eq:distexpansion-uv}) as $r\searrow0$. Moreover, the corresponding remainder terms for
$d_u$\,, and also for $d'_u$\,, can be estimated in
terms of smooth curvature expressions and are thus bounded uniformly in~$u\in S^1_p$\,.
In particular, there exists $0<\tilde\varepsilon<\varepsilon$ such that $d_u|_{[0,\tilde\varepsilon]}$ has strictly positive
derivative for each $u\in S^1_p$\,. 
Thus
$$\eta:=\min\{d_u(\tilde\varepsilon/2)\mid u\in S^1_p\}>0
$$
is a regular value of $B_{\tilde\varepsilon}(p)\ni q\mapsto\operatorname{dist}(q,\Phi(q))\in{\mathbb R}$,
so
$$\rho(u):=(d_u|_{[0,\tilde\varepsilon]})^{-1}(\eta)\in(0,\tilde\varepsilon/2]
$$
depends smoothly on~$u\in S^1_p$\,.
Let
$$\tilde U:=\{\exp_p(ru)\mid u\in S^1_p, r\in[0,\rho(u))\}.
$$
Then $\tilde U\subset U$ is an open neighborhood of~$p$, so by Remark~\ref{rem:feeling}(ii), $I(t)$ has the same
asymptotic expansion for $t\searrow0$ as
\begin{equation*}
\tilde I(t):=\int_{\tilde U}H(t,q,\Phi(q))
=\int_{S^1_p}\int_0^{\rho(u)}\bar H(t,ru,rD^\varphi(u))\cdot\ell_u(r)\,dr\,du.
\end{equation*}
Writing $d_u^{-1}$ for the inverse of $d_u|_{[0,\eta]}$ and substituting $r$ by $=d_u(r)/\sqrt t$ we obtain
\begin{equation}
\label{eq:Itildesubs}
\tilde I(t)=\int_{S^1_p}\int_0^{\eta/\sqrt t}\bar H\bigl(t,d_u^{-1}(z\sqrt t)u,d_u^{-1}(z\sqrt t)D^\varphi(u)\bigr)
\cdot\sqrt t\cdot\ell_u\bigl(d_u^{-1}(z\sqrt t)\bigr)\cdot (d_u^{-1})'(z\sqrt t)\,dz\,du.
\end{equation}
Note that
$$\operatorname{dist}\bigl(d_u^{-1}(z\sqrt t)u,d_u^{-1}(z\sqrt t)D^\varphi(u)\bigr)=z\sqrt t.
$$
Thus, $\bar H\bigl(t,d_u^{-1}(z\sqrt t)u,d_u^{-1}(z\sqrt t)D^\varphi(u)\bigr)$ for $t\searrow0$ is approximated,
uniformly in $(u,z)\in S^1_p\times[0,\eta]$, by
\begin{equation}
\label{eq:Hbarapprox}
(4\pi t)^{-1} e^{-z^2/4}\biggl(\sum_{i=0}^2{\boldsymbol u}_i(d_u^{-1}(z\sqrt t)u,d_u^{-1}(z\sqrt t)D^\varphi(u))t^i+O(t^3)\biggr).
\end{equation}
By Lemma~\ref{lem:u0bisr4},
\begin{multline*}
\sum_{i=0}^2{\boldsymbol u}_i\bigl(d_u^{-1}(z\sqrt t)u,d_u^{-1}(z\sqrt t)D^\varphi(u)\bigr)t^i=1+\frac1{12}K(p)z^2t\\
\shoveright{+\Bigl(\frac1{24C^2}\nabla^2K_p(u,u)+\frac1{160}K(p)^2-\frac1{120}\nabla^2K_p(u,u)
   \Bigr)z^4t^2+\frac13K(p)t{\text\;}}\\
\shoveright{+\Bigl(\frac1{6C^2} \nabla^2K_p(u,u)+\frac1{30}K(p)^2
  -\frac1{30}\nabla^2K_p(u,u)-\frac1{120}\Delta_gK(p)\Bigr)z^2t^2}\\
+\frac1{15}K(p)^2t^2-\frac1{15}\Delta_gK(p)t^2
 +O(t^3),
\end{multline*}
uniformly in $(u,z)\in S^1_p\times[0,\eta]$.
Moreover, from~(\ref{eq:distexpansion-uv}) one obtains
$$d_u^{-1}(s)=\frac1C s+\frac{\sin^2\varphi}{6C^5}K(p)s^3+Bs^5+O(s^7)
$$
with
$$B:=\Bigl(\frac{7\sin^4\varphi}{72C^9}+\frac{\sin^2\varphi\cdot(2-4\cos\varphi)}{90C^7}\Bigr)K(p)^2
-\frac{\sin^2\varphi\cdot(2+\cos\varphi)}{240C^7}\Delta_g K(p),
$$
and
\begin{equation*}
\begin{split}
(d_u^{-1}(s))^3&=\frac1{C^3}s^3+\frac{\sin^2\varphi}{2C^7}K(p)s^5+O(s^7),\\
(d_u^{-1}(s))^5&=\frac1{C^5}s^5+O(s^7),\\
(d_u^{-1})'(s)&=\frac1C+\frac{\sin^2\varphi}{2C^5}s^2+5Bs^4+O(s^6).
\end{split}
\end{equation*}
Using this and~(\ref{eq:ell}), one sees by a straightforward calculation:
\begin{align*}
\sqrt t\cdot\ell_u(d_u^{-1}(z\sqrt t))\cdot (d_u^{-1})'&(z\sqrt t)=\frac1{C^2}zt+
\Bigl(\frac{2\sin^2\varphi}{3C^6}-\frac1{6C^4}\Bigr)K(p)z^3t^2\\
&+\Bigl(\frac{2\sin^4\varphi}{3C^{10}}-\frac{\sin^2\varphi}{6C^8}+\frac{\sin^2\varphi\cdot(2-4\cos\varphi)}{15C^8}
 +\frac1{120C^6} \Bigr)K(p)^2 z^5 t^3\\
&+\Bigl(-\frac{\sin^2\varphi\cdot(2+\cos\varphi)}{40C^8}\Delta_gK(p)-\frac1{40C^6}\nabla^2K_p(u,u)\Bigr)z^5 t^3+O(t^4).
\end{align*}
By $2-4\cos\varphi=2C^2-2$, $2+\cos\varphi=3-\frac12 C^2$, and $\sin^2\varphi=C^2(1-\frac14 C^2)$, this becomes
\begin{multline*}
\sqrt t\cdot\ell_u(d_u^{-1}(z\sqrt t))\cdot (d_u^{-1})'(z\sqrt t)=
\frac1{C^2}zt+\Bigl(\frac1{2C^4}-\frac1{6C^2}\Bigr)K(p)z^3t^2
+\Bigl(\frac3{8C^6}-\frac1{8C^4}+\frac1{120C^2}\Bigr)K(p)^2z^5t^3\\
+\Bigl[\Bigl(-\frac3{40C^6}+\frac1{32C^4}-\frac1{320C^2}\Bigr)\Delta_gK(p)-\frac1{40C^6}\nabla^2K_p(u,u)\Bigr]z^5 t^3+O(t^4).
\end{multline*}
Multiplying this expression by~(\ref{eq:Hbarapprox}), we obtain that the integrand in~(\ref{eq:Itildesubs}) for $t\searrow0$
is approximated, uniformly in $(u,z)\in S^1_p\times[0,\eta]$, by
\begin{align*}
\frac1{4\pi} e^{-z^2/4}\cdot\biggl\{\frac1{C^2}z&+\Bigl[\Bigl(\frac1{2C^4}-\frac1{12C^2}\Bigr)z^3+\frac1{3C^2}z\Bigr]K(p)t\\
&+\Bigl[\Bigl(\frac3{8C^6}-\frac1{12C^4}+\frac1{1440C^2}\Bigr)z^5+\Bigl(\frac1{6C^4}-\frac1{45C^2}\Bigr)z^3+\frac1{15C^2}z\Bigr]K(p)^2t^2\\
&+\Bigl[\Bigl(-\frac3{40C^6}+\frac1{32C^4}-\frac1{320C^2}\Bigr)z^5-\frac1{120C^2}z^3-\frac1{15C^2}z\Bigr]\Delta_gK(p)t^2\\
&+\Bigl[\Bigl(-\frac1{40C^6}+\frac1{24C^4}-\frac1{120C^2}\Bigr)z^5+\Bigl(\frac1{6C^4}-\frac1{30C^2}\Bigr)z^3\Bigr]\nabla^2K_p(u,u)t^2
+O(t^3)\biggr\}.
\end{align*}
Recall that $\eta>0$, so for any $k\in{\mathbb N}_0$ we have $\int_{\eta/\sqrt t}^\infty e^{-z^2/4}z^k\in O(t^\infty)$ for $t\searrow0$.
Therefore, we can replace $\int_0^{\eta/\sqrt t}$ by $\int_0^\infty$ in~(\ref{eq:Itildesubs}) without changing the coefficients
in its asymptotic expansion for $t\searrow0$. Moreover,
$$\int_0^\infty e^{-z^2/4}z^{2k+1}dz=2^{2k+1}k!,
$$
giving $2$ for $k=0$, $8$ for $k=1$, and $64$ for $k=2$. Finally,
$$\int_{S^1_p}\nabla^2K_p(u,u)\,du=-\frac12\int_{S^1_p}\Delta_gK(p)\,du.
$$
Using all this, we obtain

\begin{align*}
\tilde I(t)={}&\frac{2\pi}{4\pi}\biggl\{\frac1{C^2}\cdot 2+\Bigl[\Bigl(\frac1{2C^4}-\frac1{12C^2}\Bigr)\cdot 8+\frac1{3C^2}\cdot 2\Bigr]K(p)t\\
&\hphantom{\frac{2\pi}{4\pi}\biggl\{\frac1{C^2}\cdot 2}+\Bigl[\Bigl(\frac3{8C^6}-\frac1{12C^4}+\frac1{1440C^2}\Bigr)\cdot 64+\Bigl(\frac1{6C^4}-\frac1{45C^2}\Bigr)\cdot 8+\frac1{15C^2}
\cdot 2\Bigr]K(p)^2t^2\\
&\hphantom{\frac{2\pi}{4\pi}\biggl\{\frac1{C^2}\cdot 2}+\Bigl[\Bigl(-\frac3{40C^6}+\frac1{32C^4}-\frac1{320C^2}\Bigr)\cdot 64-\frac1{120C^2}\cdot 8-\frac1{15C^2}\cdot2\Bigr]\Delta_gK(p)t^2\\
&\hphantom{\frac{2\pi}{4\pi}\biggl\{\frac1{C^2}\cdot 2}+\Bigl[\Bigl(-\frac1{40C^6}+\frac1{24C^4}-\frac1{120C^2}\Bigr)\cdot 64+\Bigl(\frac1{6C^4}-\frac1{30C^2}\Bigr)\cdot 8\Bigr]
\cdot\Bigl(-\frac12\Delta_gK(p)\Bigr)t^2\biggr\}\\
&\hphantom{\frac{2\pi}{4\pi}\biggl\{\frac1{C^2}\cdot 2}+O(t^3)\\
={}&\frac1{C^2}+\frac2{C^4}K(p)t+\Bigl[\Bigl(\frac{12}{C^6}-\frac2{C^4}\Bigr)K(p)^2-\frac2{C^6}\Delta_gK(p)\Bigr]t^2+O(t^3)
\end{align*}
for $t\searrow0$,
yielding the claimed result for the coefficient $b_2(\Phi)$ at~$t^2$ and, as an aside, the previously known formulas for
$b_0(\Phi)$ and $b_1(\Phi)$ (see Remark~\ref{rem:b0b1}).
\end{proof}

\Section{Contribution of orbisurface cone points to the second order heat coefficient}
\label{sec:orbi}

\noindent
We now consider the heat kernel of compact Riemannian orbifolds; see, e.g., \cite{DGGW} for the general framework in this context.
Let $(\mathcal O,g)$ be a closed two-dimensional Riemannian orbifold,
let $H_{\mathcal O}:(0,\infty)\times\mathcal O\times\mathcal O\to{\mathbb R}$ denote the heat kernel
associated with the Laplace operator~$\Delta_g$ on $C^\infty(\mathcal O)$, and let
$$Z(t):=\int_{\mathcal O} H_{\mathcal O}(t,x,x)\,dx
$$
be the corresponding heat trace.
It is well-known that there is an asymptotic expansion
$$Z(t)\sim(4\pi t)^{-1}\sum_{i=0}^\infty a_{i/2} t^{i/2}
$$
for $t\searrow0$; half powers may occur if $\mathcal O$ contains mirror lines. More precisely, the principal (open) stratum
contributes $(4\pi t)^{-1}\sum_{\ell=0}^\infty a^{(\mathcal O)}_\ell t^\ell$ to this expansion (where $a^{(\mathcal O)}_k$ are the integrals over~$\mathcal O$
of certain curvature invariants -- the same as in the case of manifolds), and any singular stratum $N\subset\mathcal O$
adds a contribution of the form
$$(4\pi t)^{-\operatorname{dim\,}(N)/2}\sum_{\ell=0}^\infty a^{(N)}_\ell t^\ell\,;
$$
see \cite{DGGW}, Theorem~4.8. 
In the case $N=\{\bar p\}$, where $\bar p\in\mathcal O$ is a cone point
of order~$k\in{\mathbb N}$, arising from a rotation~$\Phi$ with angle $\varphi:=2\pi/k$, 
one has $\operatorname{dim\,}(N)=0$ and
\begin{equation}
\label{eq:cone-sum}
a^{(\{\bar p\})}_\ell=\frac1k\sum_{j=1}^{k-1} b_\ell(\Phi^j),
\end{equation}
where the $b_\ell$ are as in~\ref{notrems:rot} (see \cite{DGGW}, 4.5--4.8 \& Example~5.3).
More precisely, the role of the manifold~$M$ of~\ref{notrems:rot} is played here by
the domain~$\tilde U$ of a local orbifold chart around~$\bar p$,
endowed with the pull-back of the Riemannian metric~$g$ (again denoted~$g$),
such that $(\tilde U,g)/\{\operatorname{Id},\Phi,\ldots,\Phi^{k-1}\}$
is isometric to a neighborhood of~$p$ in~$\mathcal O$; the point $p$ of~\ref{notrems:rot} is the preimage of~$\bar p$.

\begin{theorem}
\label{thm:cone}
Let $\bar p\in(\mathcal O,g)$ be a cone point of order~$k\in{\mathbb N}$ as above. Then
\begin{multline*}
a^{(\{\bar p\})}_2
=\Bigl[\frac1{2520}\Bigl(k^5-\frac1k\Bigr)+\frac1{720}\Bigl(k^3-\frac1k\Bigr)
+\frac1{180}\Bigl(k-\frac1k\Bigr)\Bigr]K(\bar p)^2\\
-\Bigl[\frac1{15120}\Bigl(k^5-\frac1k\Bigr)+\frac1{1440}\Bigl(k^3-\frac1k\Bigr)
+\frac1{180}\Bigl(k-\frac1k\Bigr)\Bigr]\Delta_gK(\bar p).
\end{multline*}
\end{theorem}

\begin{proof}
Let $p$ denote the preimage of~$\bar p$ in an orbifold chart $(\tilde U,g)$ as above.
Note that with $\varphi:=2\pi/k$ and
$C:=\sqrt{2-2\cos\varphi}$ one has
$$C^2=4\sin^2\frac\varphi2,
$$
and by \cite{CM}, p.~148 or, e.g., \cite{Uc}, 3.55,
\begin{align*}
\sum_{j=1}^{k-1}\frac1{\sin^4(j\cdot\frac\pi k)}&=\frac1{45}(k^4-1)+\frac29(k^2-1),\\
\sum_{j=1}^{k-1}\frac1{\sin^6(j\cdot\frac\pi k)}&=\frac2{945}(k^6-1)+\frac1{45}(k^4-1)+\frac8{45}(k^2-1).
\end{align*}
Combining this with (\ref{eq:cone-sum}) and Theorem~\ref{thm:b2}, we obtain
\begin{align*}
a^{(\{\bar p\})}_2
={}&\frac1k\sum_{j=1}^{k-1}\biggl[\biggl(\frac{12}{4^3\sin^6(j\cdot\frac\pi k)}-\frac2{4^2\sin^4(j\cdot\frac\pi k)}\biggr)K(p)^2
-\frac2{4^3\sin^6(j\cdot\frac\pi k)}\Delta_gK(p)\biggr]\\
={}&\frac1k\biggl\{\biggl[\frac{12\cdot2}{64\cdot 945}(k^6-1)+\Bigl(\frac{12\cdot1}{64\cdot45}-\frac{2\cdot1}{16\cdot 45}\Bigr)(k^4-1)
+\Bigl(\frac{12\cdot8}{64\cdot 45}-\frac{2\cdot2}{16\cdot 9}\Bigr)(k^2-1)\biggr]K(p)^2\\
&-\biggl[\frac{2\cdot2}{64\cdot 945}(k^6-1)
  +\frac{2\cdot1}{64\cdot 45}(k^4-1)+\frac{2\cdot8}{64\cdot 45}(k^2-1)\biggr]\Delta_gK(p)\biggr\}\\
={}&\biggl[\frac1{2520}\Bigl(k^5-\frac1k\Bigr)+\frac1{720}\Bigl(k^3-\frac1k\Bigr)+\frac1{180}\Bigl(k-\frac1k\Bigr)\biggr]K(p)^2\\
&-\biggl[\frac1{15120}\Bigl(k^5-\frac1k\Bigr)+\frac1{1440}\Bigl(k^3-\frac1k\Bigr)+\frac1{180}\Bigl(k-\frac1k\Bigr)\biggr]\Delta_gK(p).
\end{align*}
Finally, note that by definition of the curvature and the Laplacian
on Riemannian orbifolds, $K(\bar p)=K(p)$ and $\Delta_gK(\bar p)=\Delta_gK(p)$. The theorem now follows.
\end{proof}

\begin{remark}
\label{rem:cone-a0a1}
Analogously, one could derive that
\begin{align*}
a_0^{(\{\bar p\})}&=\frac1{12}\Bigl(k-\frac1k\Bigr),\\
a_1^{(\{\bar p\})}&=\biggl[\frac1{360}\Bigl(k^3-\frac1k\Bigr)+\frac1{36}\Bigl(k-\frac1k\Bigr)\biggr]K(\bar p),
\end{align*}
for an orbisurface cone point~$\bar p\in(\mathcal O,g)$ of order~$k$,
using
$$\sum_{j=1}^{k-1}\frac1{\sin^2(j\cdot\pi/k)}=\frac13(k^2-1)\text{ \ and \ }b_0(\Phi)=\frac1C\,,\;\;b_1(\Phi)=\frac2{C^2}K(p).
$$
Note that the above formulas for $a_0^{(\{\bar p\})}$ and $a_1^{(\{\bar p\})}$ were already computed in \cite{DGGW}, 5.6. 
\end{remark}

\Section{Corner contributions to the heat coefficients of geodesic polygons, up to degree two}
\label{sec:applic}

\noindent
In this section we follow ideas from \cite{Uc}, Section 4.3, concerning the case
of interior angles of the form $\gamma=\pi/k$ in geodesic polygons. However, we drop the assumption
of constant Gauss curvature which was present there and replace it by certain milder symmetricity assumptions
(see~(\ref{ass:symmetry}) below).

\begin{notation}
\label{not:polyecke}
We consider a two-dimensional Riemannian manifold $(M,g)$ again. Let $P$ 
be a compact geodesic
polygon in $(M,g)$, and let $p\in M$ be one of its corners. Let $\gamma$ be the
interior angle of~$P$ at~$p$. (For simplicity we assume that there is only one interior 
angle of~$P$ at the corner~$p$, although more general settings as considered in~\cite{Uc} 
could be treated analogously.)
As in~\ref{notrems:rot}, choose $\varepsilon_1>0$ such that $\exp_p|_{B_{\varepsilon_1}(0_p)}$ is a diffeomorphism 
onto its image
$$B:=B_{\varepsilon_1}(p).
$$
We now also assume that $\varepsilon_1$ is so small that $B\cap P$ is
the image, under $\exp_p|_{B_{\varepsilon_1}(0_p)}$\,,
of a circular sector of radius~$\varepsilon_1$ in $T_pM$.
Let $E_0\,,E_1$ be the two geodesic segments in $B\cap \partial P$ which meet at~$p$,
and let $u_0\,,u_1\in S^1_p$ be unit vectors pointing into the direction of $E_0$ and $E_1$\,, respectively.
Choose the orientation on~$B$ such that the rotation $D^\gamma:T_pM\to T_pM$ maps
$u_0$ to~$u_1$\,. Let $S:T_pM\to T_pM$ denote the reflection across ${\mathbb R} u_0$\,.
We consider the diffeomorphisms
\begin{align*}
\sigma&:=\exp_p\circ S\circ\bigl(\exp_p|_{B_{\varepsilon_1}(0_p)}\bigr)^{-1}:B\to B,\\
\delta^\gamma&:=\exp_p\circ D^\gamma\circ\bigl(\exp_p|_{B_{\varepsilon_1}(0_p)}\bigr)^{-1}:B\to B.
\end{align*}
Denote by $G$ the group of diffeomorphisms of $B$ generated by $\delta^\gamma$ and~$\sigma$.
We now assume that $\gamma$ is of the form
$$\gamma=\pi/k \text{ for some $k\ge2$ in ${\mathbb N}$, so $G$ is a dihedral group of order }4k.
$$
Moreover, we assume that, after possibly making $\varepsilon_1$ smaller,
\begin{equation}
\label{ass:symmetry}
G=\langle \{\delta^\gamma,\sigma\}\rangle \subset\operatorname{Isom}(B,g).
\end{equation}
Note that $G$ consists of the $2k$ rotations $\delta^{i\gamma}:=(\delta^\gamma)^i$  with $i\in\{0,\dots,2k-1\}$
and the $2k$ reflections $\delta^{i\gamma}\circ\sigma$.
(A special case in which the above symmetry assumptions hold is the case of $B$ being
a rotational surface with vertex~$p$.)
We choose $\varepsilon>0$ such that $\varepsilon_2:=2\varepsilon<\varepsilon_1$ and write
\begin{equation*}
\begin{gathered}
V:=B_{2\varepsilon}(p)\subset B,\quad U:=B_\varepsilon(p)\subset V,\\
W_{2\varepsilon}:=V\cap P,\quad W_\varepsilon:=U\cap P.
\end{gathered}
\end{equation*}
Finally, we denote by $H_P$, $H_V$, $H_{W_{2\varepsilon}}$ the Dirichlet heat kernels of~$P$, $V$, and $W_{2\varepsilon}$\,, respectively.
\end{notation}

\begin{remark}
\label{rem:relation}
Let the situation be as above in~\ref{not:polyecke}, and let
\begin{equation*}
Z_{W_\varepsilon}(t):=\int_{W_\varepsilon}H_P(t,q,q)\,dq,
\end{equation*}
where $dq$ abbreviates $\operatorname{\,\textit{dvol}}_g(q)$.
Note that the contribution of the interior angle at the corner~$p$ to the asymptotic
expansion of the heat trace $t\mapsto\int_P H_P(t,q,q)\,dq$ of~$P$ is the same as its contribution to the asymptotic
expansion of the function $Z_{W_\varepsilon}$ as just defined.
We will now show, using the symmetry assumption~(\ref{ass:symmetry}), that
the contribution of the interior angle~$\gamma=\pi/k$ at~$p$ to the asymptotic expansion of $Z_{W_\varepsilon}(t)$
equals $\frac12$~times the contribution of a cone point~$\bar p$ of order~$k$
to the heat kernel coefficients of a Riemannian orbisurface, where $\bar p$ has a neighborhood isometric to $B$
divided by a group of rotations about~$p$.
One could show this by using arguments analogous to those in~\cite{Uc}, p.~142--144.
We choose a related, but slightly different argument using
a little trick (see~(\ref{eq:trick}) below) involving rotations, as in the computation on p.~108 in~\cite{Uc}.

First of all, by the Principle of not feeling the boundary (recall Remark~\ref{rem:feeling}(i)), we can replace $H_P(t,q,q)$
by $H_{W_{2\varepsilon}}(t,q,q)$ in the definition of $Z_{W_\varepsilon}(t)$ without changing its asymptotic
expansion as $t\searrow0$.
Next, we describe $H_{W_{2\varepsilon}}(t,q,q)$ using Sommerfeld's method of images
(see also \cite{Uc}, Section~3.4):
For $i\in\{0,\ldots, 2k-1\}$ let
$$\sigma_i:=\delta^{i\gamma}\circ\sigma\circ\delta^{-i\gamma}\in\operatorname{Isom}(B,g)
$$
denote the reflection across the geodesic with initial vector $(D^\gamma)^i(u_0)=D^{i\gamma}(u_0)$.
Of course, $\sigma_i=\sigma_{i+k}$ for $i\in\{0,\ldots,k-1\}$.
Write
$$\Psi_i:=\sigma_i\circ\dots\circ\sigma_1\text{ for }i\in\{1,\ldots,2k-1\},\text{ \ and \ }\Psi_0:=\operatorname{Id}_V.
$$
Then
\begin{equation*}
H_{W_{2\varepsilon}}(t,q,q)=\sum_{i=0}^{2k-1}(-1)^iH_V(t,q,\Psi_i(q))
\end{equation*}
for all $t>0$ and $q\in W$. So the small-time asymptotic expansion of $Z_{W_\varepsilon}(t)$ is the same as that of
\begin{equation}
\label{eq:sommerfeld}
\sum_{i=0}^{2k-1}(-1)^i\int_{W_\varepsilon}H_V(t,q,\Psi_i(q))\,dq.
\end{equation}
We now show that sum of those summands which correspond to odd indices~$i$ does actually not enter into the corner
contribution: 
Note that $\Psi_{2j-1}=\sigma_j$ for $j\in\{1,\dots,k\}$ and thus,
using $\sigma_j=\delta^{j\gamma}\circ\sigma_0\circ \delta^{-j\gamma}$:
\begin{equation}
\label{eq:trick}
\begin{split}
\sum_{j=1}^k\int_{W_\varepsilon}&(-1)^{2j-1}H_V(t,q,\Psi_{2j-1}(q))\,dq
=-\sum_{j=1}^k\int_{W_\varepsilon}H_V(t,\delta^{-j\gamma}q,\delta^{-j\gamma}\sigma_j(q))\,dq\\
&=-\sum_{j=1}^k\int_{W_\varepsilon}H_V(t,\delta^{-j\gamma}q,\sigma_0(\delta^{-j\gamma}(q))\,dq
=-\sum_{j=1}^k\int_{\delta^{-j\gamma}(W_\varepsilon)}H_V(t,q,\sigma_0(q))\,dq\\
&=-\int_{\bigcup_{j=1,\ldots,k}\delta^{-j\gamma}(W_\varepsilon)}H_V(t,q,\sigma_0(q))\,dq
=-\int_{U'}H_V(t,q,\sigma_0(q))\,dq,
\end{split}
\end{equation}
where $U':=\bigcup_{j=1,\ldots,k}\delta^{-j\gamma}(W_\varepsilon)$ is a half-disc; $U'$ is that part of $U=B_\varepsilon(p)$
that lies on the same side of $L_\varepsilon:=\exp_p(\{ru_0\mid r\in(-\varepsilon,\varepsilon)\})$ as $\sigma_0(W_\varepsilon)=\delta^{-\gamma}(W_\varepsilon)$.
In particular, $U'$ has no corner at~$p$, and the small-time asymptotic expansion of~(\ref{eq:trick}) will
yield only the contribution of the straight boundary segment $L_\varepsilon$ to the Dirichlet heat trace expansion of the analogous
half-disc $V'\subset V$.

Write $\varphi:=2\gamma=2\pi/k$ and $\Phi:=\delta^\varphi$. Then,
on the other hand, the sum of those summands in~(\ref{eq:sommerfeld}) which correspond to even indices~$i$ gives, 
using $\Psi_{2j}=\delta^{2j\gamma}$ and the symmetry condition~(\ref{ass:symmetry}):
\begin{multline*}
\sum_{j=0}^{k-1}\int_{W_\varepsilon}(-1)^{2j}H_V(t,q,\Psi_{2j}(q))\,dq
=\frac12\cdot2\sum_{j=0}^{k-1}\int_{W_\varepsilon}H_V(t,q,\delta^{2j\gamma}(q))\,dq\\
=\frac12\sum_{j=0}^{k-1}\int_{W_\varepsilon\cup\,\delta^\gamma(W_\varepsilon)}H_V(t,q,\delta^{2j\gamma}(q))\,dq
=\frac1{2k}\sum_{j=0}^{k-1}\int_U H_V(t,q,\Phi^j(q))\,dq.
\end{multline*}
By (\ref{eq:I-asymptotic}), the asymptotic expansion for $t\searrow0$ of this sum is
$$\frac1{2k}\sum_{j=0}^{k-1}\sum_{\ell=0}^\infty b_\ell(\Phi^j)t^\ell
=\sum_{\ell=0}^\infty\alpha_\ell t^\ell\text{ \ with \ }
\alpha_\ell:=\frac1{2k}\sum_{j=0}^{k-1}b_\ell(\Phi_j).
$$
By (\ref{eq:cone-sum}), we have $\alpha_\ell=\frac12a_\ell^{(\{\bar p\})}$, where $\bar p$ is a
cone point of order~$k$ in any closed orbisurface $\mathcal O$ with the property
that some neighborhood of~$\bar p$ is isometric to $B/\{\Phi^j\mid j=0,\dots,k-1\}$.
We know the values of $\frac12a_0^{(\{\bar p\})}$, $\frac12a_1^{(\{\bar p\})}$, $\frac12a_2^{(\{\bar p\})}$
from Remark~\ref{rem:cone-a0a1} and Theorem~\ref{thm:cone}.
Finally, note that
$$k^{m-1}-\frac1k=\frac{\pi^m-\gamma^m}{\gamma^{m-1}\pi}
$$
since $\gamma=\pi/k$. So we have shown:
\end{remark}

\begin{maintheorem}
\label{mainthm:corner}
In the situation of Notation~\ref{not:polyecke}, with the symmetry assumption~(\ref{ass:symmetry}),
the contribution
of the corner $p$ with interior angle $\gamma=\pi/k$ (where $k\in{\mathbb N}$, $k\ge2$) to the asymptotic expansion
of the heat trace associated with the Dirichlet Laplacian of the geodesic polygon~$P$ has the form
$\sum_{\ell=0}^\infty c_\ell(\gamma)t^\ell$ with the coefficients $c_\ell(\gamma)$ given by
$$c_\ell(\gamma)=\frac12 a_\ell^{(\{\bar p\})},
$$
where $\bar p$ is an orbisurface cone point of order~$k$ having a neighborhood
isometric to $B/\{\delta^{2j\gamma}\mid j=0,\dots,k-1\}$. In particular,
by Theorem~\ref{thm:cone}, 
\begin{align}
\label{eq:c0}
c_0(\gamma)={}&\frac{\pi^2-\gamma^2}{24\gamma\pi},\\
\label{eq:c1}
c_1(\gamma)={}&\Bigl(\frac{\pi^4-\gamma^4}{720\gamma^2\pi}
+\frac{\pi^2-\gamma^2}{72\gamma\pi}\Bigr)K(p),\\
\label{eq:c2}
c_2(\gamma)={}&\Bigl(\frac{\pi^6-\gamma^6}{5040\gamma^5\pi}+\frac{\pi^4-\gamma^4}{1440\gamma^3\pi}
+\frac{\pi^2-\gamma^2}{360\gamma\pi}\Bigr)K(p)^2\\
&\notag-\Bigl(\frac{\pi^6-\gamma^6}{30240\gamma^5\pi}+\frac{\pi^4-\gamma^4}{2880\gamma^3\pi}
+\frac{\pi^2-\gamma^2}{360\gamma\pi}\Bigr)\Delta_gK(p).
\end{align}
(As always in this article, $\Delta_g$ here denotes ${}-\operatorname{div}_g\circ\operatorname{grad}_g$.)
\end{maintheorem}

\begin{remark}\
(i)
Formula~(\ref{eq:c0}) for $c_0(\gamma)$ seems well-known, even for general $\gamma$ (not only those of the
form $\gamma=\pi/k$) and without any symmetry assumptions; see, e.g., the discussion in~\cite{MR}.
Of course, in the case of euclidean polygons this is obvious from
the classical formula~(\ref{eq:Kac-MS}) found by D.~Ray and proved by
van den Berg and Srisatkunarajah~\cite{vdBS}.

(ii)
In the case of constant curvature~$K=1$, the above formulas (\ref{eq:c0}), (\ref{eq:c1}), (\ref{eq:c2})
-- even for general $\gamma\in(0,2\pi]$ --
were proved by Watson~\cite{Wa}. In the case of arbitrary constant curvature $K\in{\mathbb R}$
the same was proved by U\c car in~\cite{Uc}, the main breakthrough there being the computation
of the Green kernel for an arbitrary geodesic wedge in the hyperbolic plane.
Those authors actually computed $c_\ell(\gamma)$
for {\it every\/} $\ell\in{\mathbb N}_0$ in the case $K=1$, resp.~$K\in{\mathbb R}$ constant.
It turns out that for constant curvature~$K$, one has $c_\ell(\gamma)=f_\ell(\gamma)\cdot K^\ell$ for certain
rational functions~$f_\ell$\,.
Of course, based on U\c car's and Watson's formulas,
the above formula for $c_1(\gamma)$, as well as  the coefficient at $K(p)^2$ in~$c_2(\gamma)$,
was to be expected. However, the constant curvature case did not provide insight into the way
in which $\Delta_gK(p)$ -- which is, up to linear combinations, the only other curvature invariant of order four
in dimension two besides $K(p)^2$ -- might enter into $c_2(\gamma)$.

(iii)
To the author's best knowledge, formula~(\ref{eq:c2}) for $c_2(\gamma)$ (with $\gamma\in\{\pi/k\mid k\in{\mathbb N}\}$
and under the symmetry assumptions~(\ref{ass:symmetry})), especially its coefficient at $\Delta_gK(p)$,
was not known previously. In particular, the main theoretic insight that
this formula provides is that here the
coefficient at $\Delta_gK(p)$ is a rational function of~$\gamma$,  and that it is of a similar structure
as the coefficient at $K(p)^2$. We expect that the formula extends to general $\gamma\in(0,2\pi]$; see Conjecture~\ref{conj}.
below.

(iv)
Note that the symmetry condition~(\ref{ass:symmetry}), which has been necessary for our approach,
implies that the gradient of $K$ at~$p$ vanishes. Therefore,
the methods of the present article cannot lead in any way, in situations where that symmetry condition is absent,
to any knowledge about the possible coefficient of $\|\nabla K(p)\|^2$ in~$c_3(\gamma)$ (note that $\|\nabla K(p)\|^2$
is one of the curvature invariants of order six).
Concerning $c_2(\gamma)$, however, we expect that formula~(\ref{eq:c2}) from the above theorem holds more generally, at
least if $\nabla^2K_p$ still is a multiple of~$g_p$\,.
So we conclude this paper with the following conjecture:
\end{remark}

\begin{conjecture}
\label{conj}
Let $\gamma\in(0,2\pi]$, and let $P$ be a compact geodesic polygon in a two-dimensional Riemannian manifold $(M,g)$.
Let $p$ be a corner of~$P$ with interior angle~$\gamma\in(0,2\pi]$, and assume that $\nabla^2 K_p$ is a multiple of $g_p$\,.
Then the coefficient at~$t^2$
in the small-time asymptotic expansion of the Dirichlet heat kernel of~$P$ is given by formula~(\ref{eq:c2}).
\end{conjecture}

\appendix

\Section{Proof of Lemma~\ref{lem:distexpansion}}

\noindent
We partly follow Nicolaescu's approach from \cite{Ni}, Appendix~A. He considered Riemannian
manifolds of arbitrary dimension~$n$ and there derived the expansion of $\operatorname{dist}(\exp_p(x),\exp_p(y))$
up to order four.
In dimension two, his formula corresponds to the first two terms of formula (\ref{eq:distexpansion}),
with $K(p)\|x\wedge y\|^2$ replaced by
$\langle R(x,y)y,x\rangle $. The idea in~\cite{Ni} is to use the fact that for any $q\in M$, the function
$f:=\operatorname{dist}^2(q,\,.\,):M\to{\mathbb R}$ satisfies, wherever it is smooth (in particular, near~$q$),
a so-called Hamilton-Jacobi equation: 
\begin{equation}
\label{eq:hamjac}
\|df\|^2=4f.
\end{equation}
Here we use $\|\,.\,\|$ to denote the pointwise norm
canonically induced by~$g$ on tensor fields, and we will do similarly for $\langle\,\,,\,\rangle$.

Choose a small neighborhood~$W$ of~$0\in T_pM$ contained in the domain of injectivity of~$\exp_p$
and such that $U:=\exp_p(W)\subset M$ is convex (meaning that for all $q,w\in U$, there exists a unique
geodesic in~$M$ with length~$\operatorname{dist}(q,w)$, and that geodesic is contained in~$U$).
Consider
$$F:W\times W\ni(x,y)\mapsto\operatorname{dist}^2(\exp_p(x),\exp_p(y))\in{\mathbb R}.
$$
We write the Taylor expansion of~$F$ at~$(0,0)$ in the form
\begin{equation}
\label{eq:taylor}
(T^\infty_{(0,0)}F)(x,y)=(F_0+F_1+F_2+F_3+\ldots)(x,y)
\text{ \ with } F_m=F_{m,0}+F_{m-1,1}+\ldots+F_{0,m}\,,
\end{equation}
where each $F_{k,\ell}(x,y)$ is $k$-linear in~$x$ and $\ell$-linear in~$y$.
Since $F$ is symmetric, $F_{\ell,k}$ is obtained from $F_{k,\ell}$ by interchanging $x$ and~$y$.
Moreover,
$$F(x,0)=\|x\|^2,\text{ \ hence }F_{k,0}=0=F_{0,k}\text{ for all }k>2.
$$
Note that by the First Variation Formula we have
$$\frac d{dt}\Bigl|_{t=0}\,F(tx,y)=-2\langle x,y\rangle,\text{ \ hence }
F_{1,k}=0=F_{k,1}\text{ for all }k>1.
$$
(This was not used in~\cite{Ni}.) In particular,
$$F_3=0\text{ \  and \ } F_4=F_{2,2}
$$
(as already known),
and what we are actually after are explicit formulas, in our two-dimensional setting, for
\begin{equation*}
F_5=F_{3,2}+F_{2,3}\text{ \ and \ }
F_6=F_{4,2}+F_{3,3}+F_{2,4}.
\end{equation*}
For each $y\in W$, $F^y:=F(\,.\,,y):W\to{\mathbb R}$ is smooth.
Let~$\hat g$ be the Riemannian metric $(\exp_p|_W)^*g$ on~$W$.
Then (\ref{eq:hamjac}) says
$$4F^y=\|dF^y\|_{\hat g}^2.
$$
Since we assume $\operatorname{dim\,} M=2$, we can
express $\|(dF^y)_x\|_{\hat g}^2$ at each nonzero $x\in W$ as follows: Consider
the $\hat g$-orthonormal basis
$\{x/\|x\|,\tilde x/\|\tilde x\|_{\hat g}\}$
of $T_x W$, where $\tilde x\in T_x W$ denotes the $90$-degree rotation of~$x$
with respect to an arbitrarily chosen orientation on the euclidean plane~$(T_pM,g_p)$. Then
\begin{equation}
\label{eq:transform}
\begin{split}
\|(dF^y)_x\|_{\hat g}^2={}&(dF^y)_x(x)^2/\|x\|^2+(dF^y)_x(\tilde x)^2/\|\tilde x\|_{\hat g}^2\\
    ={}&\bigl((dF^y)_x(x)^2+(dF_y)_x(\tilde x)^2\bigr)/\|x\|^2-(dF^y)_x(\tilde x)^2/\|x\|^2
		+(dF^y)_x(\tilde x)^2/\|\tilde x\|^2_{\hat g}\\
    ={}&\|(dF^y)_x\|^2+(dF^y)_x(\tilde x)^2(\|\tilde x\|_{\hat g}^{-2}-\|x\|^{-2})
\end{split}
\end{equation}
For this, recall that $\|\,.\,\|$ denotes the norm with respect to $g_p$, and for $x$ viewed as an element of $T_x W$,
$\|x\|_{\hat g}=\|x\|$ since $\exp_p$ is a radial isometry.
Using~(\ref{eq:theta}) for $u=x/\|x\|$, $r=\|x\|$ and noting that
$$\|(d\exp_p)_x(\tilde x)\|=\theta_u(r)\|\tilde x\|=\theta_u(r)\|x\|,
$$
we have, for $\tilde x$ viewed as an element of $T_x W$:
\begin{equation*}
\|\tilde x\|_{\hat g}=\|x\|-\frac16 K(p)\|x\|^3-\frac1{12} dK_p(x)\|x\|^3
+\frac1{120}K(p)^2\|x\|^5
-\frac1{40}\nabla^2K_p(x,x)\|x\|^3+O(\|x\|^6).
\end{equation*}
By the resulting expansion of $\|\tilde x\|_{\hat g}^{-2}$ and (\ref{eq:transform}), equation (\ref{eq:hamjac}) becomes
\begin{align*}
4F(x,y)={}&\|(dF^y)_x\|^2+{}\\
      &+(dF^y)_x(\tilde x)^2\cdot\Bigl(\frac 13 K(p)+\frac16 dK_p(x)
			+\frac1{15} K(p)^2\|x\|^2+\frac1{20}\nabla^2K_p(x,x)
			+O(\|x\|^3)\Bigr).
\end{align*}
Comparing the terms of total order five in $x$ and $y$ in this equation we get, writing
$F_m^y:=F_m(\,.\,,y)$, noting that $(dF_m^y)_x$ is of total
order $m-1$, and recalling $F_0=0$, $F_1=0$,
$F_2(x,y)=\|x-y\|^2$, $F_3=0$:
\begin{equation*}
\begin{split}
4F_5(x,y)&=2\langle (dF_5^y)_x,(dF_2^y)_x\rangle +(dF_2^y)_x(\tilde x)^2\cdot\frac16 dK_p(x)\\
&=4(dF_5^y)_x(x-y)+4\langle x-y,\tilde x\rangle ^2\cdot\frac16 dK_p(x).
\end{split}
\end{equation*}
In particular, by $(dF_{k,\ell}^y)_x(x)=k F_{k,\ell}(x,y)$ we have
$$4F_{3,2}(x,y)=12 F_{3,2}(x,y)+4\|x\wedge y\|^2\cdot\frac16 dK_p(x),
$$
which gives
$$F_{3,2}(x,y)=-\frac 1{12}dK_p(x)\|x\wedge y\|^2.
$$
The claimed form of
$F_5$ now follows by symmetry in $x$ and $y$.
Similarly, taking the well-known formula
$$F_4(x,y)=-\frac13K(p)\|x\wedge y\|^2
$$ for granted
(which could otherwise first been proved analogously), and using
\begin{equation*}
(dF_2^y)_x(\tilde x)=-2\langle \tilde x,y\rangle,\text{ \ }
(dF^y_4)_x(\tilde x)=\frac23 K(p)\langle x,y\rangle \langle \tilde x,y\rangle,\text{ \ }
\langle \tilde x,y\rangle ^2=\|x\wedge y\|^2,
\end{equation*}
we obtain
\begin{equation*}
\begin{split}
4F_6(x,y)={}&\|(dF_4^y)_x\|^2+2\langle (dF_6^y)_x,(dF_2^y)_x\rangle
         +2(dF^y_2)_x(\tilde x) (dF_4^y)_x(\tilde x)\cdot \frac13 K(p)\\
         &+ (dF_2^y)_x(\tilde x)^2\cdot\Bigl(\frac1{15} K(p)^2\|x\|^2+\frac1{20}\nabla^2 K_p(x,x)\Bigr)\\
				 ={}&\frac49 K(p)^2\|x\wedge y\|^2\|y\|^2+4(dF_6^y)_x(x-y)
				 -\frac89 K(p)^2\|x\wedge y\|^2\langle x,y\rangle\\
				 &+4\|x\wedge y\|^2\cdot\Bigl(\frac1{15} K(p)^2\|x\|^2+\frac1{20}\nabla^2 K_p(x,x)\Bigr).
\end{split}
\end{equation*}
In particular,
$$4 F_{4,2}(x,y)=16 F_{4,2}(x,y)+4\|x\wedge y\|^2\cdot\Bigl(\frac1{15} K(p)^2\|x\|^2+\frac1{20}\nabla^2 K_p(x,x)\Bigr).
$$
Thus,
$$F_{4,2}(x,y)=\|x\wedge y\|^2\cdot\Bigl(-\frac1{45}K(p)^2\|x\|^2-\frac1{60}\nabla^2K_p(x,x)\Bigr),
$$
and the analogous expression for $F_{2,4}(x,y)$, as claimed. Finally,
\begin{equation*}
\begin{split}
4 F_{3,3}(x,y)={}&4(dF_{4,2}^y)_x(-y)+4(dF_{3,3}^y)_x(x)-\frac89 K(p)^2\|x\wedge y\|^2\langle x,y\rangle\\
   ={}&\|x\wedge y\|^2\cdot\Bigl(\frac8{45}K(p)^2\langle x,y\rangle +\frac8{60}\nabla^2K_p(x,y)\Bigr)+12 F_{3,3}(x,y)
	 -\frac89 K(p)^2\|x\wedge y\|^2\langle x,y\rangle,
\end{split}
\end{equation*}
yielding
$$F_{3,3}(x,y)=\|x\wedge y\|^2\cdot\Bigl(\frac4{45} K(p)^2\langle x,y\rangle -\frac1{60}\nabla^2K_p(x,y)\Bigr),
$$
as claimed.

\end{document}